\documentclass[amstex,10pt, leqno]{amsart}
\usepackage{amsthm,amsmath,amsfonts,amssymb,enumerate,graphicx,hyperref,multicol,color,textcomp,hyperref,graphics}
\usepackage{hyperref}
\usepackage{doi}
\usepackage{float} 
\usepackage{cite}

\newtheorem{lem}{Lemma}
\newtheorem{thm}{Theorem}

\newtheorem{note}{Note}
\newtheorem{corollary}{Corollary}
\newtheorem{rem}{Remark}

{\qed\bigskip}

{\thanks
{
\begin{footnotesize}
\hspace*{-7mm}
 \\
 $^*$Corresponding author
\end{footnotesize}
}}

\begin{document}

\leftline{ \scriptsize \it  }
\title []
{Simultaneous approximation properties of a certain
genuine hybrid type operator}
\maketitle
\begin{center} Nidhi Bisht$^1$ and  Asha Ram Gairola$^*$ \\
$^{1,2}$Department of Mathematics, Doon University, Dehradun-248001, India\\ bishtnidhi12@gmail.com, ashagairola@gmail.com
\end{center}

\begin{abstract}
We obtain estimates in simultaneous approximation for a summation-integral type genuine hybrid operator. The convergence of derivatives of operator to the corresponding derivatives of the functions is proved and estimates for rate of approximation are obtained.
\end{abstract}
\section{Introduction}
The  classical Sz$\acute{a}$sz-Mirakian Operators   devised and studied independently by Mirakian \cite{mirakyan1941approximation} and  Sz$\acute{a}$sz \cite{szasz1950generalization} are given as
\begin{equation*}\label{Sm}V_n(f,x)=\sum\limits_{k=0}^{\infty} f\left(\frac{k}{n}\right)\frac{e^{-nx}(nx)^k}{k!}.
\end{equation*}
Through the years extensive study has been done on approximation properties of these operators. Several modifications of the Sz$\acute{a}$sz-Mirakian Operators have also been introduced. One such modification was established by Prasad et al. \cite{prasad1983approximation} where he considered Baskakov basis function as the weight function and subsequently defined it as\\
\begin{equation*}\label{Smbopr}
G_n(f;x)=(n-1)\sum_{k=0}^{\infty}\frac{e^{-nx}(nx)^{k}}{k!}\int\limits_0^{\infty}f(t)\binom{n+k-1}{k}\frac{t^k}{(1+t)^{n+k}}\,dt.
\end{equation*}

P\u{a}lt\v{a}nea\cite{Paltanea} studied a two parameter integral operator with modified Sz\'{a}sz--Mirakyan basis in summation.
 Gupta et al.\cite{gupta2018approximation} further generalized operators $G_n(f;x)$ by taking a general form of two basis functions which are modified Sz\'{a}sz and Baskakov function and referred to them as genuine hybrid operator. For function $f:[0,\infty)\rightarrow \mathbb{R}$ and parameters $\rho>0,$ $c\in(0,1)$ the operators in \cite{gupta2018approximation}  are given by
\begin{equation}\label{Hybrid Operators}
B^{\rho}_{\alpha}(f;x,c)=\sum\limits_{k=1}^{\infty} p_{\alpha,k}(x,c)\int\limits_{0}^{\infty}\theta^{\rho}_{\alpha,k}(t)f(t)dt+p_{\alpha,0}(x,c)f(0),
\end{equation}
where $$p_{\alpha,k}(x,c)=\frac{(-x)^k}{k!}\phi^{(k)}_{\alpha,c}(x),\theta^{\rho}_{\alpha,k}(t)=\frac{\alpha\rho}{\Gamma(k\rho)}e^{-\alpha\rho t}(\alpha\rho t)^{k\rho-1}.$$
The authors of \cite{gupta2018approximation} established direct approximation theorems in terms of modulus of smoothness and proved point-wise convergence for derivatives of the operator $B^{\rho}_{\alpha}(f;x,c).$


The simultaneous approximation properties i.e. the study of convergence behaviour of derivatives of a linear positive operator for a function with given smoothness is an interesting part of approximation theory.
Recently, Tang et al.\cite{tang2020weighted} explored positive and inverse results for weighted simultaneous approximation of linear combinations of Baskakov operator using jacobi weight function.   It is worth mentioning Mohammed et al.\cite{Alkhaled} wherein the authors studied the basic convergence theorem in simultaneous approximation and obtained Voronovskaja-type asymptotic formula.
Lately, Abel\cite{abel2023simultaneous} studied simultaneous approximation by operators of exponential type and proved that under certain conditions, asymptotic expansions for sequences of operators belonging to slightly larger class of operators can be differentiated term by term.
Mohammed et.al.\cite{Mohammed_2021} illustrated the simultaneous approximation by a modification family of the classical sequence of Sz\'{a}sz operators with parameter $s > \frac{-1}{2}.$ For further literature in this direction we refer to the interesting papers \cite{abel2022simultaneous}, \cite{acar2015simultaneous}, \cite{cheng2012simultaneous}, \cite{article},\cite{gupta1995convergence}, \cite{sharma2020simultaneous},\cite{sinha2024simultaneous},\cite{tachev2014approximation}, \cite{verma2012convergence,xie2007pointwise,xie2005pointwise} and the references therein.

The aim of this paper is to study a special case of the operator (\ref{Hybrid Operators}) which is defined as
\begin{equation}\label{defoperator}
L_{n,c}f(x)=\sum_{k=0}^{\infty}p_{n,k}(x)\int_{0}^{\infty}\theta_{n,k}(t)f(t)\,dt
\end{equation}
$p_{n,k}(x)=\frac{(n/c)_k (c x)^k}{k! (c x+1)^{\frac{n}{c}+k}},$  $\theta_{n,k}(t)=\frac{n e^{-n x} (n x)^k}{k!}$ and
$(n/c)_k=\prod _{i=0}^{k-1} \left(\frac{n}{c}+i\right).$\\

The  direct results in simultaneous approximation are obtained in local setting.  For global results, we use the properties of Steklov means to obtain error estimates in terms of modulus of smoothness. The Steklov averaging means are defined as follows.\\

Let $f\in C[a,b]$ and $[a_1,b_1]\subset (a,b).$ Then,
for sufficiently small $h>0,$ the Steklov mean $f_{h,s}$
of $s$-th order corresponding to $f$ is defined by
\begin{equation*}f_{h,s}(t)=h^{-s}\int\limits_{-h/2}^{h/2}\!\!...\!\!\int\limits_{-h/2}^{h/2}
\left(f(t)+(-1)^{s-1}\Delta^s_{\sum_{i=1}^{s}t_i}f(t)\right)\prod_{i=1}^{s}dt_i
,t\in [a_1,b_1].\end{equation*}

\begin{lem}\label{lemma0.4} For the function
$f_{h,s}$, we have
\begin{enumerate}[(a)]
\item  $f_{h,s}$ has derivatives up to order $s$ over
$[a_1,b_1]$;
\item $\big\|f^{(r)}_{h,s}\big\|_{[a_1,b_1]}\leqslant
C h^{-r}\omega_s(f,h,[a,b]),r=1,2,...,s;$
\item $\|f-f_{h,s}\|_{[a_1,b_1]}\leqslant
C \omega_s(f,h,[a,b])$;
\item $\|f_{h,s}\|_{[a_1,b_1]}\leqslant C \|f\|_{[a,b]};$
\item $\big\|f^{(s)}_{h,s}\big\|_{[a_1,b_1]}\leqslant
C h^{-s} \|f\|_{[a,b]},$
\end{enumerate}
where $C$ is a certain constant not same at the each occurrence that depends on $s$
but is independent of $f$ and $h.$
\end{lem}
\begin{proof}Following \cite{hewitt2012real}  or \cite{timan2014theory},  the proof follows easily.
\end{proof}


\section{Preliminaries}

\begin{lem}\label{moments}
If $M_{n,i}(x)=L_{n,c}(t^i,x)$ and $\mu_{n,i}(x)=L_{n,c}((t-x)^i,x),$ then,
\begin{enumerate}
\item $M_{n,0}(x)=1$
\item $M_{n,1}(x)=\frac{n x+1}{n}$
\item $M_{n,2}(x)=\frac{c n x^2+n^2 x^2+4 n x+2}{n^2}$
\item $M_{n,3}(x)=\frac{n x \left(2 c^2 x^2+9 c x+18\right)+3 n^2 x^2 (c x+3)+n^3 x^3+6}{n^3}$
\item $M_{n,4}(x)=\frac{n^2 x^2 \left(11 c^2 x^2+48 c x+72\right)+2 n x \left(3 c^3 x^3+16 c^2 x^2+36 c x+48\right)+2 n^3 x^3 (3 c x+8)+n^4 x^4+24}{n^4}$
\end{enumerate}
\end{lem}

\begin{lem}\label{recurrencerelation}Let $\mu_{n,i}(x)=L_{n,c}((t-x)^i,x).$ Then, there holds the recurrence
\begin{equation}\label{id:recurrencerelation2}\mu_{n,m+1}(x)=\frac{1}{n}\Big(x(1+cx)\left[\mu'_{n,m}(x)+m\mu_{n,m-1}(x)\right]+(m+1)\mu_{n,m}(x)+mx\mu_{n,m-1}(x)\Big)
\end{equation}
\end{lem}
\begin{proof}
We have $x(1+cx)\frac{d}{dx}p_{n,k}(x)=(k-nx)p_{n,k}$ and $\frac{d}{dt}\theta_{n,k}(x)=(\frac{k}{t}-n)\theta_{n,k}(t).$
Making use of the identity $k-nx=(\frac{k}{t}-n)(t-x)+(\frac{k}{t}-n)x+n(t-x)$ in the identity
\begin{equation*}
\mu'_{n,m}(x)=\sum_{k=0}^{\infty}\int_{0}^{\infty}\left(p'_{n,k}(x)(t-x)^m-m(t-x)^{m-1}p_{n,k}(x)\right)\theta_{n,k}(t)f(t)\,dt
\end{equation*}
we get
\begin{align*}
&x(1+cx)\left[\mu'_{n,m}(x)+m\mu_{n,m-1}(x)\right]\\
&=\sum_{k=0}^{\infty}p_{n,k}(x)\int_{0}^{\infty}\left(\left(\frac{k}{t}-n\right)\left[(t-x)^{m+1}+x(t-x)^m\right]+n(t-x)^{m+1}\right)\theta_{n,k}(t)\,dt\\
&=\sum_{k=0}^{\infty}p_{n,k}(x)\int_{0}^{\infty}\left[(t-x)^{m+1}+x(t-x)^m\right]\theta'_{n,k}(t)\,dt+n\mu_{n,m+1}(x).\\
\end{align*}
Integration by parts implies
\begin{align*}
&x(1+cx)\left[\mu'_{n,m}(x)+m\mu_{n,m-1}(x)\right]\\
&=-\sum_{k=0}^{\infty}p_{n,k}(x)\int_{0}^{\infty}\left(m+1)(t-x)^m+mx(t-x)^{m+1}\right)\theta_{n,k}(t)\,dt+n\mu_{n,m+1}(x)\\
&=-(m+1)\mu_{n,m}(x)-mx \mu_{n,m-1}(x)+n\mu_{n,m+1}(x).
\end{align*}
This completes the proof.
\end{proof}
\begin{corollary}\label{orderofmoments}For $x\in [0,\infty)$ and $n\geq m,$ using lemma\ref{recurrencerelation}, we obtain
\begin{enumerate}
\item $\mu_{n,0}(x)=1$
\item $\mu_{n,1}(x)=\frac{1}{n}$
\item $\mu_{n,2}(x)=\frac{c n x^2+2 n x+2}{n^2}$
\item $\mu_{n,3}(x)=\frac{n x \left(2 c^2 x^2+9 c x+12\right)+6}{n^3}$
\item $\mu_{n,4}(x)=\frac{2 n x \left(3 c^3 x^3+16 c^2 x^2+36 c x+36\right)+3 n^2 x^2 (c x+2)^2+24}{n^4}$
\end{enumerate}
It follows by an easy induction that $\mu_{n,m}(x)$ is a polynomial in $x$ of degree at most $m.$ Further, we have
\[\mu_{n,m}(x)=O\left(n^{\lfloor \frac{m+1}{2}\rfloor}\right).\]
\end{corollary}
\begin{lem}\label{mgf}Let $M_{n,r}(x)=L_{n,c}(t^r,x).$ Then,
\begin{equation}\label{directmoment}M_{n,r}(x)=n^{-r}\sum_{j=0}^{r}\binom{r}{j}(1+cx)^j(n/c)_j(-1)^{r-j}\prod_{k=1}^{r-j}(n/c-k).\end{equation}
\end{lem}
\begin{proof}
We have
\begin{align*}
L_{n,c}(e^{\theta t},x)&=\sum_{k=0}^{\infty}p_{n,k}(x)\int_{0}^{\infty}\frac{n}{k!}(nt)^ke^{-(n-\theta)t}\,dt\\
&=\sum_{k=0}^{\infty}p_{n,k}(x)\left(\frac{n}{n-\theta}\right)^{k+1}\\
&=\left(\frac{n(n-\theta)^{n/c-1}}{(n-\theta(1+cx))^{n/c}}\right).
\end{align*}
Now, (\ref{directmoment}) follows by the linearity of the operator $L_{n,c}(f,x).$
By Leibniz rule we get
\begin{align*}
M_{n,r}(x)&=\frac{\partial^r}{\partial \theta^r}\left(\frac{n(n-\theta)^{n/c-1}}{(n-\theta(1+cx))^{n/c}}\right)\Bigg|_{\theta=0}\\
&=n\sum_{j=0}^{r}\binom{r}{j}\frac{\partial^j}{\partial \theta^j}(n-\theta(1+cx))^{-n/c}\frac{\partial^{r-j}}{\partial \theta^{r-j}}
(n-\theta)^{n/c-1}\Bigg|_{\theta=0}\\
&=n\sum_{j=0}^{r}\binom{r}{j}(1+cx)^j(n/c)_jn^{-r-1}(-1)^{r-j}\prod_{k=1}^{r-j}(n/c-k)\Bigg|_{\theta=0}\\
&=n^{-r}\sum_{j=0}^{r}\binom{r}{j}(1+cx)^j(n/c)_j(-1)^{r-j}\prod_{k=1}^{r-j}(n/c-k).\end{align*}

\end{proof}

\begin{corollary}\label{expressioncentralmoments}For the functions $\mu_{n,r}(x), r\in N_0,$ we have the formula
\begin{equation}\label{directcentralmoment}\mu_{n,r}(x)=\sum_{r=0}^{s}\sum_{j=0}^{r}\binom{s}{r}\binom{r}{j}(-1)^{s-j}n^{-r}(1+cx)^j(n/c)_j
\prod_{k=1}^{r-j}(n/c-k).\end{equation}
The functions  $\mu_{n,r}(x)$ and $M_{n,r}(x), r\in N_0,$ are polynomials of degree $r$ and can be expressed in the powers of
$(1+cx).$ \end{corollary}

%

 \begin{lem}\label{lemma5,hoa} There exist the polynomials $q_{i,j,r}(x)$ independent of
 $n$ and $\nu$ such that
 \begin{equation*}\frac{d^r}{dt^r}p_{n,k}(x)=\sum_{2i+j\leqslant
r\atop i,j\geqslant0}n^i(k-n x)^j\frac{q_{i,j,r}(x)}{x^r(1+cx)^r}p_{n,k}(x).
\end{equation*}\end{lem}

\begin{lem}\label{lemma4,hoa} Let $\delta$ and $\gamma$ be any two positive real numbers. Then, for any
 $s>0$ and  $0<a<b<\infty,$ we have,
 \[\sup\limits_{x\in [a,b]}\Bigg|\sum_{k=0}^{\infty}p_{n,k}(x)
\int\limits_{|t-x|\geq \delta}\theta_{n,k}(t)e^{\gamma t}\,dt\Bigg|=O\big(n^{-s}\big).\]
 \end{lem}

\begin{lem}\label{derivativetransfer}Let $f\in C^{r}[0,\infty).$  Then,
\begin{equation}\label{id:derivativetransfer}
\frac{d^r}{dw^r}L_{n,c}(f,w)\Big|_{w=x}=\sum_{k=0}^{\infty}p_{n+r,k}(x)\int_{0}^{\infty}\theta_{n,k+r}(t)f^{(r)}(t)\,dt
\end{equation}
\end{lem}

\begin{proof} Since, \[\frac{d}{dx}p_{n,k}(x)=n\left(p_{n+1,k-1}(x)-p_{n+1,k}(x)\right)\] and \[\frac{d}{dt}\theta_{n,k}(t)=n\left(\theta_{n,k-1}(t)-\theta_{n,k}(t)\right),\] for $r=1,$ we have
\begin{align*}
&\frac{d}{dw}L_{n,c}(f,w)\Big|_{w=x}\\
&=\sum_{k=0}^{\infty}n\left(p_{n+1,k-1}(x)-p_{n+1,k}(x)\right)\int_{0}^{\infty}\theta_{n,k+r}(t)f^{(r)}(t)\,dt\\
&=\sum_{k=0}^{\infty}np_{n+1,k}(x)\int_{0}^{\infty}\left(\theta_{n,k+1}(t)-\theta_{n,k}(t)\right)f(t)\,dt\\
&=-\sum_{k=0}^{\infty}p_{n+1,k}(x)\int_{0}^{\infty}\theta'_{n,k+1}(t)f(t)\,dt.\\
\end{align*}
Integration by parts implies
\begin{align*}
\frac{d}{dw}L_{n,c}(f,w)\Big|_{w=x}&=\sum_{k=0}^{\infty}p_{n+1,k}(x)\int_{0}^{\infty}\theta_{n,k+1}(t)f'(t)\,dt.
\end{align*}
Now, by straightforward calculations and principal of mathematical induction, the proof follows.
\end{proof}
\begin{note} For simultaneous approximation properties of the operator $L_{n,c}(f,x),$ we use the transformed operator $L_{n,c,r}(f,x).$ This operator is defined by
\begin{equation}\label{associatedoperator}
  L_{n,c,r}(f,x)=\sum_{k=0}^{\infty}p_{n+r,k}(x)\int_{0}^{\infty}\theta_{n,k+r}(t)f(t)\,dt.
\end{equation}
Clearly $L_{n,c,r}\left(f^{(r)},x\right)=\frac{d^r}{dw^r}L_{n,c}(f,w)\Big|_{w=x},$ whenever (\ref{id:derivativetransfer}) is meaningful.
\end{note}
Now onwards we denote the sequence $(n/c)_s c^n n^{-s}$ by $\lambda_n(c,s).$

%
%
%
\section{Main Results}
\begin{thm}\label{Convergence}
 Let    $f\in C^s[0,\infty)$ and $x\in[0,\infty)$ then
\begin{equation*} \lim_{n\rightarrow \infty}\frac{1}{\lambda_n(c,s)}\frac{d^s}{dw^s}L_{n,c}(f;w)\Big|_{w=x}=f^{(s)}(x).
\end{equation*}
\end{thm}
\begin{proof}
Writing
\[f(t)=\sum_{k=0}^{s} \frac{f^{(k)}(x)}{k!}(t-x)^k+(t-x)^s\epsilon(t,x)\] and operating $L_{n,c}(f;\cdot)$ we get
\begin{align*}
L_{n,c}(f;w)&=\sum_{k=0}^{s} \frac{f^{(k)}(x)}{k!}L_{n,c}((t-x)^k,w)+L_{n,c}((t-x)^s\epsilon(t,x),w).
\end{align*}
Since $L_{n,c}((t-x)^k,w)$ are polynomials in $w$ of degree at most $k,$
\begin{align*}
\frac{d^s}{dw^s}L_{n,c}(f;w)\Big|_{w=x}&=\frac{f^{(s)}(x)}{s!}\frac{d^s}{dw^s}L_{n,c}((t-x)^s,w)\Big|_{w=x}+\frac{d^s}{dw^s}L_{n,c}((t-x)^s\epsilon(t,x),w)\Big|_{w=x}.
\end{align*}
The coefficient of $w^r$ in $L_{n,c}(t^r,w)$ is equal to
\[n\frac{\partial^r}{\partial \theta^r}(n-\theta(1+cw))^{-n/c}(n-\theta)^{n/c-1}\Bigg|_{\theta=0}=n^{-r}(n/c)_r(1+cw)^r.\]
Thus
\begin{align}\label{coeffws}
\frac{d^s}{d w^s}L_{n,c}((t-x)^s,w)\Big|_{w=x}&=\frac{d^s}{d w^s}\left(\sum_{r=0}^{s}(-x)^{s-r}\binom{s}{r}L_{n,c}(t^r,w)\right)\Big|_{w=x}\nonumber
\\
&=\sum_{r=0}^{s}(-x)^{s-r}\binom{s}{r}n^{-r}(n/c)_r\frac{d^s}{d w^s}(1+cw)^r\Big|_{w=x}\nonumber\\
&=s!\lambda_n(c,s).
\end{align}

As the operator $L_{n,c}$ are bounded and linear, it follows by standard arguments that $\lim_{n\rightarrow \infty}\frac{d^s}{dw^s}L_{n,c}((t-x)^s\epsilon(t,x),w)\Big|_{w=x}=0.$
The first term is therefore
\begin{equation*}
\frac{f^{(s)}(x)}{s!}\frac{d^s}{dw^s}L_{n,c}((t-w)^s,w)\Big|_{w=x}=f^{(s)}(x)\lambda_n(c,s).
\end{equation*}
It is plain that
\[\lim_{n\rightarrow \infty}\frac{1}{\lambda_n(c,s)}\left(\frac{f^{(s)}(x)}{s!}\frac{d^s}{dw^s}L_{n,c}((t-w)^s,w)\Big|_{w=x}\right)=f^{(s)}(x).\]
\end{proof}

\begin{rem} If, $\lim_{n\rightarrow \infty} c^sn^{-s}(n/c)_s=1,$ then we get the limit
\begin{equation*} \lim_{n\rightarrow \infty}\frac{d^s}{dw^s}L_{n,c}(f;w)\Big|_{w=x}=f^{(s)}(x).
\end{equation*}
\end{rem}

\begin{thm}\label{smoothfunction}
 Let    $f\in C^{s+2}[0,\infty)$ and $x\in[0,\infty)$ then
\begin{align*}&\lim_{n\rightarrow \infty}n\Bigg(\frac{1}{\lambda_n(c,s)}\frac{d^s}{dw^s}L_{n,c}(f,w)\Big|_{w=x}-f^{(s)}(x)\Bigg)\\
&= (c    x +  s +  1)f^{(s+1)}(x)+\frac{x(2+cx)}{2}f^{(s+2)}(x).
\end{align*}
\end{thm}
\begin{proof}

We write
\[f(t)=\sum_{k=0}^{s+2} \frac{f^{(k)}(x)}{k!}(t-x)^k+(t-x)^s\epsilon(t,x)\] and operate by $L_{n,c}(f;\cdot)$ both side to get
\begin{align*}
L_{n,c}(f;w)&=\sum_{k=0}^{s+2} \frac{f^{(k)}(x)}{k!}L_{n,c}((t-x)^k,w)+L_{n,c}((t-x)^s\epsilon(t,x),w)\\
\end{align*}
Since, $L_{n,c}((t-x)^k,w)$ are polynomials of degree at most $k,$ we have
\begin{align*}
\frac{d^s}{dw^s}L_{n,c}(f;w)\Big|_{w=x}&=\sum_{k=s}^{s+2} \frac{f^{(k)}(x)}{k!}\frac{d^s}{dw^s}L_{n,c}((t-x)^k,w)\Big|_{w=x}+\frac{d^s}{dw^s}L_{n,c}((t-x)^s\epsilon(t,x),w)\Big|_{w=x}\\
&=E_1+E_2.
\end{align*}
It is sufficient to calculate coefficients of $w^k$ for $k=s,s+1$ and $k=s+2$ only. By (\ref{coeffws}), we have that coefficient of $w^s$ in
$L_{n,c}((t-x)^s,w)$ is $\lambda_n(c,s).$ So that $L_{n,c}((t-x)^s,w)=\lambda_n(c,s) w^s+O(w^{s-1}).$ Therefore,
\begin{equation}\label{k=sterm}
\frac{d^s}{dw^s}L_{n,c}((t-x)^s,w)\Big|_{w=x}=s! \lambda_n(c,s).
\end{equation}
Next, we calculate $\frac{d^s}{dw^s}L_{n,c}((t-x)^{s+1},w)\Big|_{w=x}$ as follows. We have
\begin{align*}
L_{n,c}((t-x)^{s+1},w)&=L_{n,c}(t^{s+1},w)-(s+1)x L_{n,c}(t^s,w)+O(w^{s-1}),\\
&=-(s+1)x\left(\lambda_n(c,s)  w^s\right)+M_{n,s+1}(w)+O(w^{s-1}).\end{align*}
Using (\ref{directmoment}),
\begin{align*}
\frac{d^s}{dw^s} M_{n,s+1}(w)&=\frac{d^s}{dw^s}\Bigg[n\Bigg\{(s+1)\frac{\partial^s}{\partial \theta^s}(n-\theta(1+cw))^{-n/c} \frac{\partial}{\partial \theta}(n-\theta)^{n/c-1}\Bigg|_{\theta=0}\\
&+(n-\theta)^{n/c-1}\frac{\partial^{s+1}}{\partial \theta^{s+1}}(n-\theta(1+cw))^{-n/c} \Bigg|_{\theta=0}\Bigg\}+O(w^{s-1})\Bigg]\\
&=\frac{\lambda_n(c,s)}{n}(s+1)!\left((n+cs)(w+1/c)-n/c+1\right) .
\end{align*}
Therefore, we have
\begin{align}\label{k=s+1term}
\frac{f^{(s+1)}(x)}{(s+1)!}\frac{d^s}{dw^s}L_{n,c}((t-x)^{s+1},w)\Big|_{w=x}&=\lambda_n(c,s)(1+s(1+cx))\frac{f^{(s+1)}(x)}{n}.
\end{align}
Finally, we calculate $\frac{d^s}{dw^s}L_{n,c}((t-x)^{s+2},w)\Big|_{w=x}.$ In the identity
\[\frac{d^s}{dw^s}L_{n,c}((t-x)^{s+2},w)=\sum_{j=0}^{s+2}\binom{s+2}{j}(-x)^{s+2-j}\frac{d^s}{dw^s}L_{n,c}(t^j,w)\] we need to compute three terms $E_j's$ corresponding to
$j=s,s+1,s+2$ respectively.
Now, following the steps of (\ref{k=sterm}), we get
\[L_{n,c}(t^s,w)=\binom{s+2}{2}x^2\lambda_n(c,s) w^s+O(w^{s-1})\] so that
\begin{equation}\label{Es}E_s=\binom{s+2}{2}x^2\lambda_n(c,s)s!,\end{equation}
\begin{align}\label{Es+1}E_{s+1}&=-\frac{f^{(s+2)}(x)}{(s+2)!}(s+2)x\frac{d^s}{dw^s}L_{n,c}((t-x)^{s+1},w)\Big|_{w=x}\nonumber\\
&=-\frac{f^{(s+2)}(x)}{(s+2)!}(s+2)x\lambda_n(c,s)(s+1)!\left(\frac{(n+cs)(w+1/c)-n/c+1}{n}\right).
\end{align}
And for $E_{s+2}$ we observe
\begin{align}\label{Es+2}
E_{s+2}&=\frac{f^{(s+2)}(x)}{(s+2)!}\frac{d^s}{dw^s}L_{n,c}((t-x)^{s+2},w)\Big|_{w=x}\nonumber\\
&=\frac{f^{(s+2)}(x)}{(s+2)!}n\frac{d^s}{dw^s}\Bigg[\frac{\partial^{s+2}}{\partial \theta^{s+2}}(n-\theta(1+cw))^{-n/c}(n-\theta)^{n/c-1}\nonumber\\
&+(s+2)\frac{\partial^{s+1}}{\partial \theta^{s+1}}(n-\theta(1+cw))^{-n/c}\frac{\partial}{\partial \theta}(n-\theta)^{n/c-1}\nonumber\\
&+\frac{(s+1)(s+1)}{2}\frac{\partial^{s}}{\partial \theta^{s}}(n-\theta(1+cw))^{-n/c}\frac{\partial^2}{\partial \theta^2}(n-\theta)^{n/c-1}\Bigg]_{w=x}\nonumber\\
&=\frac{f^{(s+2)}(x)}{(s+2)!}\frac{\lambda_n(c,s)}{n^{s+2}}\frac{d^s}{dw^s}\Bigg[(n/c+s+1)(n/c+s)(1+cw)^{s+2}\nonumber\\
&-(s+2)(n/c+s)(1+cw)^{s+1}(n/c-1)\nonumber\\
&+\frac{(s+2)(s+1)}{2}(n/c-1)(n/c-2)(1+cw)^s\Bigg]_{w=x}\nonumber\\
&=\frac{f^{(s+2)}(x)\lambda_n(c,s)}{n^2}\Bigg[(n/c+s+1)(n/c+s)\frac{(1+cx)^2}{2}\nonumber\\
&-(n/c+s)(n/c-1)(1+cx)+\frac{(n/c-1)(n/c-2)}{2}\Bigg].\end{align}
Combining (\ref{Es})-(\ref{Es+2}), it follows that
\begin{align*}
&\frac{d^s}{dw^s}L_{n,c}((t-x)^{s+2},w)\Big|_{w=x}\\
&=\lambda_n(c,s)f^{(s)}(x)+\frac{\lambda_n(c,s)}{n}f^{(s+1)}(x)\left(1+s(1+cx)\right)\\
&+\frac{f^{(s+2)}(x)}{2n^2}\left(2+s(1+cx)(2+cx)+s^2(1+cx)^2+nx(2+cx)\right).
\end{align*}
Therefore, we have
\begin{align*}
&\lim_{n\rightarrow \infty}n\Bigg(\frac{1}{\lambda_n(c,s)}\frac{d^s}{dw^s}L_{n,c}(f,w)\Big|_{w=x}-f^{(s)}(x)\Bigg)=(c    x +  s +  1)f^{(s+1)}(x)+x\left(1+\frac{cx}{2}\right)f^{(s+2)}(x).
\end{align*}

This completes the proof.
\end{proof}
\begin{rem}\label{remvorno}For $s=0,$ we obtain the following result
\begin{align*}
\lim_{n\rightarrow \infty}n\Bigg(\frac{1}{\lambda_n(c,0)}L_{n,c}(f,x)-f(x)\Bigg)= (c    x  +  1)f'(x)+x\left(1+\frac{cx}{2}\right)f''(x).
\end{align*}
Further, if $c\uparrow 1,$ then the operator $L_{n,c}(f,x)$ reduces to the Baskakov-Sz\'{a}sz type operator. And we get
\begin{align*}
\lim_{n\rightarrow \infty}n\left(L_{n,c}(f,x)-f'(x)\right)= (x  +  1)f^{(s+1)}(x)+\frac{x(x+2)}{2}f''(x).
\end{align*}

\end{rem}

Next theorem provide an upper bound in simultaneous approximation in global setting.

\begin{thm}\label{th1Simultaneous}Let $[a_1,b_1]\subset(a,b)\subset (0,\infty),$ and $f^{(r)}$ exists and is continuous on $[a,b]$ then for sufficiently large $n,$
\begin{equation*}\big\|L_{n,c}^{(r)}\left(f,\cdot\right)-f^{(r)}\big\|_{[a_1,b_1]}\leqslant
C \left(n^{-1}\|f\|+\omega_2\big(f^{(r)};n^{-1/2};[a_1,b_1]\big)\right),\end{equation*}
where $C$ is independent of $f$ and $n.$
\end{thm}
 \begin{proof}We have that
 \begin{align*}&\big\|L_{n,c}^{(r)}\left(f,\right)-f^{(r)}\big\|_{[a_1,b_1]}\\
 &\leqslant
\big\|L_{n,c}^{(r)}\left(f-f_{h,2},\right)\big\|_{[a_1,b_1]}+
\big\|L_{n,c}^{(r)}\left(f_{h,2},\right)-f_{h,2}^{(r)}\big\|_{[a_1,b_1]}
\\&+\big\|f^{(r)}-f^{(r)}_{h,2}\big\|_{[a_1,b_1]}=A_1+A_2+A_3,\  \textrm{say.}
\end{align*}
Since, $\left(f_{h,2}\right)^{(r)}=\left(f^{(r)}\right)_{h,2},$ by property of  the Steklov mean it follows that
\[A_3\leqslant C\,\omega_{2}\left(f^{(r)},h,[a_1,b_1]\right).\]
Next, applying Theorem \ref{smoothfunction} and the Riesz-Thorin interpolation theorem,  we have that
\begin{align*}A_2&\leqslant C\,n^{-1}\sum_{m=r}^{r+2}\Big\|f^{(m)}_{h,2}\Big\|_{[a_1,b_1]}\\
&\leqslant C\,n^{-1}\left(\Big\|f_{h,2}\Big\|_{[a_1,b_1]}+
\Big\|f^{(r+2)}_{h,2}\Big\|_{[a_1,b_1]}\right)\\
&\leqslant C\,n^{-1}\left(\Big\|f_{h,2}\Big\|_{[a_1,b_1]}+
\Big\|\left(f^{(r)}\right)''_{h,2}\Big\|_{[a_1,b_1]}\right).\end{align*}
Hence, by properties of Steklov mean we have
\[A_2\leqslant C\,n^{-1}\left(\|f\|+h^{-2}\omega_{2}\big(f^{(r)},h,[a_1,b_1]\big)\right).\]
In order to find estimation for $A_3,$ let us write $f-f_{h,2}=\Psi.$ By our hypothesis, we can write
\begin{eqnarray*}\Psi(t)&=&\sum_{m=0}^{r}\frac{\Psi^{(m)}(x)}{m!}(t-x)^m+
\frac{\Psi^{(r)}(\xi)-\Psi^{(r)}(x)}{r!}(t-x)^r\chi(t)
\\&+& h(t,x)\left(1-\chi(t)\right),\end{eqnarray*}
where $\xi$ lies between $t$ and $x$, and $\chi$ is the
characteristic function of the interval $(a,b).$ For
$t\in(a,b)$ and $x\in [a_1,b_1],$ we have
\begin{eqnarray*}\Psi(t)&=&\sum_{m=0}^{r}\frac{\Psi^{(m)}(x)}{m!}
(t-x)^m+
\frac{\Psi^{(r)}(\xi)-\Psi^{(r)}(x)}{r!}(t-x)^r,
 \end{eqnarray*}
If $t\in[0,\infty)\setminus (a,b), x\in [a_1,b_1]$ then,
we define
\begin{eqnarray*}h(t,x)=\Psi(t)
-\sum_{m=0}^{r}\frac{\Psi^{(m)}(x)}{m!}(t-x)^m.\end{eqnarray*}
Now,
\begin{eqnarray*}\frac{d^r}{dw^r}L_{n,c}\big(\Psi(t),w\big)\Big|_{w=x}&=&\sum_{m=0}^{r}\frac{\Psi^{(m)}(x)}{m!}
\frac{d^r}{dw^r}L_{n,c}\big((t-x)^m,w\big)\Big|_{w=x}\\
&+&
\frac{d^r}{dw^r}\frac{L_{n,c}\big(\big(\Psi^{(r)}(\xi)-\Psi^{(r)}(x)\big)(t-x)^r\chi(t),w\big)\Big|_{w=x}}{r!}
\\&+& \frac{d^r}{dw^r}L_{n,c}\big(h(t,x)\left(1-\chi(t)\right),w\big)\Big|_{w=x}:=B_1+B_2+B_3,\ \ \textrm{say.}\end{eqnarray*}
By using the identity (\ref{k=sterm}), we write that
\begin{align*}B_1&=\sum_{m=0}^{r}\frac{\Psi^{(m)}(x)}{m!}
\frac{d^r}{dw^r}L_{n,c}\big((t-x)^m,x\big)\Big|_{w=x}\\
&=\sum_{m=0}^{r}\frac{\Psi^{(m)}(x)}{m!}\sum_{l=0}^{m}{m\choose l}(-x)^{m-l}
\frac{d^r}{dw^r}L_{n,c}\big(t^l,x\big)\Big|_{w=x}\\
&= \lambda_n(c,r)\Psi^{(r)}(x).
\end{align*}
Hence,
\[|B_1|\leqslant C\,\lambda_n(c,r)\|f^{(r)}-f^{(r)}_{h,2}\|_{[a_1,b_1]}.\]
Next, using Lemma \ref{lemma5,hoa} and Schwarz's inequality for integration and then for summation we get
\begin{align*}|B_2|&\leqslant\frac{2}{r!}\|f^{(r)}-f^{(r)}_{h,2}\|_{[a_1,b_1]}
\frac{d^r}{dw^r}L_{n,c}\big(\chi(t)|t-x|^r,w\big)\Big|_{w=x}\\
&\leqslant\frac{2}{r!}\|f^{(r)}-f^{(r)}_{h,2}\|_{[a_1,b_1]}\sum_{2i+j\leqslant r\atop i,j\geqslant0}n^i
\frac{|q_{i,j,r}(x)|}{x^r(1+x)^r}n\sum_{k=0}^{\infty}p_{n,k}(x)|k-nx|^j\times\\
&\times
\int\limits_{0}^{\infty}\theta_{n,k}(t)\chi(t)|t-x|^r\,dt\\
&\leqslant\frac{2}{r!}\|f^{(r)}-f^{(r)}_{h,2}\|_{[a_1,b_1]}\sum_{2i+j\leqslant r\atop i,j\geqslant0}n^{i+1}
\frac{|q_{i,j,r}(x)|}{x^r(1+x)^r}\sum_{k=0}^{\infty}p_{n,k}(x)|k-nx|^j\times\\
&\times
\Bigg(\int\limits_{0}^{\infty}\theta_{n,k}(t)\,dt\Bigg)^{\!\!1/2}
\Bigg(\int\limits_{0}^{\infty}\theta_{n,k}(t)(t-x)^{2r}\,dt\Bigg)^{\!\!1/2}\end{align*}

\begin{align*}&\leqslant C\,\|f^{(r)}-f^{(r)}_{h,2}\|_{[a_1,b_1]}\sum_{2i+j\leqslant r\atop i,j\geqslant0}n^i
\Bigg(\sum_{k=0}^{\infty}p_{n,k}(x)(k-nx)^{2j}\Bigg)^{\!\!1/2}\times\\
& \times\Bigg(n\sum_{k=0}^{\infty}p_{n,k}(x)\int\limits_{0}^{\infty}
\theta_{n,k}(t)(t-x)^{2r}\Bigg)^{\!\!1/2}\\
&\leqslant C\,\|f^{(r)}-f^{(r)}_{h,2}\|_{[a_1,b_1]}\sum_{2i+j\leqslant r\atop i,j\geqslant0}n^i
O\big(n^{j/2}\big)O\big(n^{-r/2}\big)\\
&\leqslant C\,\|f^{(r)}-f^{(r)}_{h,2}\|_{[a_1,b_1]}.
\end{align*}
Since $t\in(0,\infty)\setminus(a,b),$ we can choose a $\delta>0$ in
such a way that $|t-x|\geqslant\delta$ for all $x\in [a_1,b_1].$
Thus, by Lemma \ref{lemma5,hoa}, we obtain
\begin{align*}|B_3|&\leqslant \sum_{2i+j\leqslant r\atop i,j\geqslant0}n^i\frac{|q_{i,j,r}(x)|}{x^r(1+x)^r}n\sum_{k=0}^{\infty}p_{n,k}(x)|k-nx|^j
\int\limits_{|t-x|\geqslant\delta}\theta_{n,k}(t)|h(t,x)|\,dt
\end{align*}
For $|t-x|\geqslant\delta,$ we can find a constant $C>0$ such that $|h(t,x)|\leqslant C|t-x|^s.$ Finally using Schwarz's inequality for integration and then for integration,  and Lemma \ref{lemma4,hoa}, it easily follows that    $B_3=O\big(n^{-s}\big)$ for any $s>0.$

Combining the estimates $B_1-B_3,$ we obtain in view of (c) of Steklov mean,
\begin{eqnarray*}A_3&\leqslant&C\,\|f^{(r)}-f^{(r)}_{h,2k+2}\|_{[a_1,b_1]}+O\big(n^{-s}\big)\\
&\leqslant&C\,\omega_{2k+2}\big(f^{(r)},h,[a_1,b_1]\big)+O\big(n^{-s}\big).\end{eqnarray*}
Taking $h=n^{-1/2}$ the proof is completed.\end{proof}

Our last result proves local rate of approximation.

\begin{thm}
Let  $f^{(r)}\in C[0,\infty).$  Then
\[\left|\frac{d^r}{d w^r}L_{n,c}(f;w)\Big|_{w=x}-f^{(r)}(x)\right|\leq 2\,\omega\left(f^{(r)},\frac{q_n(x,r)}{n}\right),\]
where $q_n(x,r)=\left(n x (c x+2)+r (x (c x+4)+3)+r^2 (x+1)^2+2\right)^{1/2}.$
\end{thm}
\begin{proof}
We have
\[\frac{d^r}{d w^r}L_{n,c}(f;w)\Big|_{w=x}=\sum\limits_{k=0}^{\infty}P_{n+r,k}(x)\int\limits_{0}^{\infty}\theta_{n,k+r}(t)f^{(r)}(t)\,dt\]
Since $\sum\limits_{k=0}^{\infty}P_{n+r,k}(x)=1$ and $\int\limits_{0}^{\infty}\theta_{n,k+r}(t)dt=1.$
Therefore
\begin{align*}
\left|\frac{d^r}{d w^r}L_{n,c}(f;w)\Big|_{w=x}-f^{(r)}(x)\right|&=\left|\sum\limits_{k=0}^{\infty}P_{n+r,k}(x)\int\limits_{0}^{\infty}\theta_{n,k+r}(t)\left(f^{(r)}(t)dt-f^{(r)}(x)\right)\right|\\
&= L_{n,c,r}(|f^{(r)}(t)-f^{(r)}(x)|;x).\\
\end{align*}
In view of the inequality
\begin{equation*}
|f^{(r)}(t)-f^{(r)}(x)|\leq  \left(1+\frac{|t-x|}{\delta}\right)\omega\left(f^{(r)};\delta\right),
\end{equation*}
we have
\[\left|\frac{d^r}{d w^r}L_{n,c}(f;w)\Big|_{w=x}-f^{(r)}(x)\right|\leq \omega\left(f^{(r)};\delta\right)\left(L_{n,c,r}(1;x)+\frac{L_{n,c,r}(|t-x|;x)}{\delta}\right).\]
Using Schwarz's inequality, we obtain
\[L_{n,c,r}(|t-x|;x)\leq (L_{n,c,r}((t-x)^2);x)^{\frac{1}{2}}.\]
\end{proof}
By direct calculations, we have
\[L_{n,c,r}(t-x)^2);x)=\frac{n x (c x+2)+r (x (c x+4)+3)+r^2 (x+1)^2+2}{n^2}.\]
Hence,
\begin{align*}\left|\frac{d^r}{d w^r}L_{n,c}(f;w)\Big|_{w=x}-f^{(r)}(x)\right|\leq \omega\left(f^{(r)};\delta\right)\left(1+\frac{(L_{n,c,r}((t-x)^2);x)^{\frac{1}{2}}}{\delta}\right).
\end{align*}
The theorem now follows by
choosing \[\delta=\left(\frac{n x (c x+2)+r (x (c x+4)+3)+r^2 (x+1)^2+2}{n^2}\right)^{1/2}.\]
\begin{rem}
If $f^{(r)}\in \text{Lip}_M \alpha,$ then
\begin{align*}\left|\frac{d^r}{d w^r}L_{n,c}(f;w)\Big|_{w=x}-f^{(r)}(x)\right|&\leq 2M \frac{\left(n x (c x+2)+r (x (c x+4)+3)+r^2 (x+1)^2+2\right)^{\alpha/2}}{n^{\alpha}}\\
&\leq M\frac{C(r,x)}{n^{\alpha/2}}.
\end{align*}
\end{rem}

\bibliographystyle{siam}
\bibliography{Nidhi}

\end{document}